\documentclass[a4paper]{gtart}

\usepackage[inner=25mm, outer=25mm, textheight=225mm]{geometry}

\usepackage{psfrag, graphicx, epstopdf, pstricks, subfigure, epsfig,color}

\usepackage{amsmath, amssymb}
\usepackage{algorithm}
\usepackage{algorithmic}
\usepackage{mathrsfs,url}

\usepackage[small,bf]{caption}
\newtheorem{Theorem}[equation]{Theorem}
\newtheorem{Proposition}[equation]{Proposition}
\newtheorem{Lemma}[equation]{Lemma}

\theoremstyle{remark}

\newcommand{\R}{\mathbb{R}}
\newcommand{\RP}{\mathbb{R} P^2}

\newcommand{\tri}{\mathcal{T}}

\DeclareMathOperator{\SL}{SL}
\DeclareMathOperator{\SO}{SO}
\DeclareMathOperator{\PSL}{PSL}
\DeclareMathOperator{\PGL}{PGL}

\begin{document}

\title{An algorithm for the Euclidean cell decomposition\\ of a cusped strictly convex projective surface}
\author{Stephan Tillmann and Sampson Wong}
\begin{abstract}
Cooper and Long generalised Epstein and Penner's Euclidean cell decomposition of cusped hyperbolic $n$--manifolds of finite volume to non-compact strictly convex projective $n$--manifolds of finite volume. We show that Weeks' algorithm to compute this decomposition for a hyperbolic surface generalises to strictly convex projective surfaces.
\end{abstract}


\makeshorttitle

%%%%%%%%%%%%%%%%%%%%%
%%%%%%%%%%%%%%%%%%%%%

\section{Introduction}
The Epstein-Penner decomposition~\cite{epstein1988euclidean} is an elegant yet powerful construction in the study of non-compact hyperbolic manifolds of finite volume. Penner~\cite{penner1987decorated} uses it to assign to each point of the decorated Teichm\"uller space an ideal cell decomposition of the surface and proves the remarkable result that this assignment induces a cell decomposition of the decorated Teichm\"uller space. 

Cooper and Long~\cite{cooper2013generalization} recently generalised Epstein and Penner's construction to obtain Euclidean cell decompositions of all non-compact, strictly convex projective manifolds of finite volume (and simply call this an \emph{Epstein-Penner decomposition}), and point out that this can be used to define a decomposition of the moduli space of such structures on a manifold. This raises the following questions: Is there an algorithm to compute the Epstein-Penner decomposition? Are all components of the decomposition of the moduli space cells? Which other results about the Epstein-Penner decomposition in the hyperbolic setting generalise to the strictly convex projective setting?

This paper addresses the first question for surfaces by generalising an edge flipping algorithm due to Weeks~\cite{weeks1993convex}, which computes the Epstein-Penner decomposition of a cusped hyperbolic surface. Edge flipping algorithms were first used by Lawson \cite{lawson1977software} to compute Delauney triangulations. Such algorithms use a sequence of local modifications to arrive at a globally optimal solution, and decisions on which edge to flip come from purely local considerations. They also work well in computing convex hulls of finite point sets in $\mathbb{R}^3$ (see \cite{gao2013flip}). However, in this paper we are concerned with convex hulls of infinite point sets. The new algorithm is presented in \S\ref{edge flipping}, and the proof of correctness given in \S\ref{sec:correctness}. Since hyperbolic geometry is a subgeometry of projective geometry, the proposed algorithm is still applicable to cusped hyperbolic surfaces, and may be a useful tool for further study of strictly convex projective structures on surfaces and their deformations. For instance, in \cite{HT2015} it is used to show that the decomposition of the moduli space of the once-punctured torus is indeed a cell decomposition. Software for our algorithm was implemented by the second author in {\tt sage}\rm~\cite{sage} to produce the examples given in the last section.

%%%%%%%%%%%%%%%%%%%%%
%%%%%%%%%%%%%%%%%%%%%

\section{The edge flipping algorithm}

%%%%%%%%%%%%%%%%%%%%%

\subsection{Cooper and Long's Construction}
\label{cooper long construction}

We summarise the construction and results due to Cooper and Long~\cite{cooper2013generalization} in the case of surfaces. Let $\Omega$ be a strictly convex domain in the real projective plane and suppose $S = \Omega / \Gamma$ is a strictly convex projective surface of finite volume with $k \geq 1$ cusps. Since there is an analytic isomorphism $\PGL(3, \R)\cong \SL(3,\R),$ we may assume $\Gamma < \SL(3,\R).$ The $(\SL(3, \R), \RP)$--structure of $S$ lifts to a $(\SL(3, \R), \mathbb{S}^2)$--structure, and we denote a lift of $\Omega$ to $\mathbb{S}^2\subset \R^3$ by $\Omega^+.$
A \emph{light-cone representative} of $p \in \partial \Omega$ is a lift $v_p \in \mathcal L = \mathcal L^+ = \R^+ \cdot \partial \Omega^+.$ 

Each cusp $c$ of $S$ corresponds to an orbit of parabolic fixed points on $\partial \Omega.$ Choose an orbit representative $p_c \in \partial \Omega,$ and hence a light-cone representative $v_{c} = v_{p_c} \in \mathcal L.$ The set $B = \{ \Gamma \cdot v_c \mid c \text{ is a cusp of } S\}$ is discrete. Let $C$ be the convex hull of $B.$ Then the projection of the faces of $\partial C$ onto $\Omega$ is a $\Gamma$--invariant cell decomposition of $\Omega,$ and hence descends to a cell decomposition of $\Omega/\Gamma,$ called an \emph{Epstein-Penner decomposition} by Cooper and Long. Varying the light-cone representatives $v_c$ gives a $(k-1)$--parameter family of $\Gamma$--invariant cell decompositions of $\Omega$. In particular, the decomposition of the surface $\Omega/\Gamma$ is canonical if $k = 1.$

%%%%%%%%%%%%%%%%%%%%%

\subsection{Ideal triangulations of surfaces}

The following facts are well known (see, for instance, Lackenby \cite{lackenby2000taut} for the first two). The second is not needed to prove existence and correctness of our algorithm; we give a proof using the algorithmic construction of the Epstein-Penner decomposition. Let $S_{g,k}$ be a closed orientable surface of genus $g$ with $k$ marked points, and let $S$ denote the complement of the set $P$ of marked points. We will always assume that $S$ has negative Euler characteristic, whence $2g+k>2.$ An \emph{essential arc} in $S$ is the intersection with $S$ of an arc embedded in $S_{g,k}$ that has endpoints in $P,$ interior disjoint from $P$ and is not homotopic (relative to $P$) to a point in $S_{g,k}.$ An \emph{ideal triangulation} $\tri$ of $S$ is a union of pairwise disjoint essential arcs that are pairwise non-homotopic. The components of $S\setminus \tri$ are \emph{ideal triangles}, and we regard two ideal triangulations of $S$ as equivalent if they are isotopic via an isotopy of $S_{g,k}$ that fixes $P$.

\begin{Lemma}
\label{ideal triangulation}
The surface $S$ admits an ideal triangulation. Moreover, every ideal triangulation has $-2 \chi(S)$ ideal triangles.
\end{Lemma}

\begin{proof}
First suppose $g\ge 1.$ It is well-known that $S_{g,k}$ has a (singular) triangulation with a single vertex $v$ and such that no edge is null-homotopic and no two edges are homotopic. We may assume that $v\in P$ and that all edges are disjoint from the remaining points in $P.$ Given $v\neq w \in P,$ we can divide the triangle containing $w$ into three triangles with vertices in $\{v, w\}$ by adding three arcs not meeting $P$ in their interiors. Each of these arcs is essential in $S$ (since it connects distinct points in $P$) and no two are homotopic relative to $P$ (since this is true for the arcs in the boundary of the triangle). So by construction, the resulting triangulation of $S_{g,k}$ gives an ideal triangulation of $S_{g,k} \setminus \{v, w\}.$ This procedure can now be iterated to give an ideal triangulation of $S = S_{g,k}\setminus P.$ The number of triangles follows from $\chi(S) = \chi(S_{g,k}) - k.$

In the case $g=0$ we have $k\ge 3$, so as a starting point one can take a triangle on $S_{0,k}$ with vertices on 3 pairwise distinct points of $P,$ and then apply the same procedure as above.
\end{proof}

An \emph{edge flip} on an ideal triangulation consists of picking two distinct ideal triangles sharing an edge, removing the shared edge to form a square, and dividing this square along its other diagonal.

For instance, in the case of the once-punctured torus any two ideal triangulations are made up of three essential arcs and divide the once-punctured torus into two ideal triangles. All of these ideal triangulations are combinatorially equivalent. However, performing an edge flip results in a non-isotopic ideal triangulation. The space of isotopy classes of ideal triangulations of the once-punctured torus naturally inherits the structure of the infinite trivalent tree, where vertices correspond to isotopy classes of ideal triangulations, and there is an edge between two such classes if and only if they are related by an edge flip. A well known geometric realisation of this was described by Floyd and Hatcher~\cite{FH}. In general, we have:

\begin{Lemma}
\label{finite sequence of elementary moves}
Any two ideal triangulations of $S$ are related by a finite sequence of edge flips. 
\end{Lemma}

\begin{proof}
The surface $S= S_{g,k}\setminus P$ has a complete hyperbolic structure of finite volume since its Euler characteristic is negative. The Epstein-Penner construction provides a canonical ideal cell decomposition of $S$. Below Theorems~\ref{algorithm correctness}  and \ref{algorithm finite} imply that any ideal triangulation of $S$ can be modified into the canonical cell decomposition by a finite number of edge flips and deleting a finite number of redundant edges. Since deleting edges is not an elementary move, the two theorems imply that any ideal triangulation of $S$ is related to the canonical ideal cell decomposition but with its polygonal cells triangulated. But any two triangulations of a polygon are also related by a finite number of edge flips. Hence, any two ideal triangulations are related to the canonical cell decomposition plus redundant edges by a finite sequence of edge flips, so are related to each other.
\end{proof}

Now suppose $S$ has a strictly convex real projective structure of finite volume, giving identifications $\pi_1(S) = \Gamma < \SL(3,\R)$ and $S = \Omega / \Gamma,$ where $\Omega \subset \RP$ is a strictly convex set.  An ideal triangulation of $S$ is \emph{straight} if each ideal edge is the image of the intersection of $\Omega$ with a projective line.

\begin{Lemma}
\label{geodesic ideal triangulation}
Every ideal triangulation of $S$ is isotopic to a straight ideal triangulation.
\end{Lemma}

\begin{proof}
The ideal triangulation $\tri$ of $S$ can be lifted to a $\Gamma$--equivariant topological ideal triangulation $\widetilde{\tri}$ of $\Omega.$ There is a homeomorphism of $\Omega$ that fixes the boundary $\partial \Omega$ and takes each topological ideal edge to a segement of a projective line. Since $\overline{\Omega}$ is a closed disc, this homeomorphism is isotopic to the identity. Whence the topological ideal triangulation $\widetilde{\tri}$ of $\Omega$ is isotopic to a straight ideal triangulation. Since the set of ideal endpoints of edges is $\Gamma$--equivariant and any two such endpoints determine a unique segment of a projective line in $\Omega,$ the straight ideal triangulation is $\Gamma$--equivariant. We may therefore choose a $\Gamma$--equivariant isotopy between the ideal triangulations of $\Omega,$ and push this down to an isotopy of $S.$
\end{proof}

%%%%%%%%%%%%%%%%%%%%%%%%%%
%%%%%%%%%%%%%%%%%%%%%%%%%%

\subsection{The algorithm}
\label{edge flipping}

The crucial difference between the proposed algorithm and Weeks' is the method of detection of convex angles. Weeks' tilt formula \cite{weeks1993convex} and its generalisations \cite{sakuma1995generalized, ushijima2002tilt} rely on the Minkowski norm whereas the proposed algorithm uses the standard Euclidean metric on $\R^3.$ The new method is to check the following property of a convex hull: if $F$ is a face of the convex hull $C$, all other vertices of the polyhedron are on the same side of $X_F$, where $X_F$ is the plane through $F$. In the case where $C$ is the convex hull of the $\Gamma$--invariant discrete subset $B$ on a light-cone $\mathcal L$, the plane $X_F$ separates $\R^3$ into two components. If $F$ is a face of $C$, then the interior of the polyhedron $C$ lies entirely in one component of $\R^3\backslash X_F$, and the other component contains the origin.

Define the \emph{neighbouring faces} of $F$ to be the faces which share an edge with $F$. Define the vertex of a neighbouring face which is not on the shared edge to be a \emph{neighbouring vertex} of $F$. Let $X_F$ be the plane passing through $F$. If $v$ and the origin lie in the same component of $\R^3 \backslash X_F$, then we say $v$ is \emph{below} the face $F$. If $v$ and the origin lie in different components then $v$ is \emph{above} the face $F$. Otherwise $v$ is on the plane $X_F$ and is \emph{coplanar} with $F$.

Define an \emph{edge flip} on $F$ and a neighbouring vertex $v$ to be the edge flip which removes the common edge. If the vertices of $F$ are $\{a,b,c\}$ where $ab$ is the common edge, then the edge flip creates the two new faces $\{c,v,a\}$ and $\{c, b, v\}$. We call an edge flip \emph{admissible} if $v$ is below $F$.

We call $F$ \emph{locally convex} if each neighbouring vertex $v$ is either above $F$ or coplanar with $F$. Equivalently, $F$ is locally convex if there are no admissible edge flips which include face $F$.

Let $\Omega$ be a strictly convex domain in $\RP$. The edge flipping algorithm is as follows. Let the projective surface be $\Omega/\Gamma,$ where $\Gamma \subset \SL(\Omega)$ is freely acting, discrete and finitely generated. We start with an arbitrary cell decomposition of $\Omega/\Gamma$ into geodesic ideal triangles, its existence is ensured by Lemma~\ref{geodesic ideal triangulation}. The ideal triangulation projects to a $\Gamma$--invariant polyhedron with $\Gamma$--invariant vertices on the light-cone $\mathcal L$.

For a face $F$ on the $\Gamma$--invariant polyhedron, we call $\Gamma \cdot F$ the \emph{face class} of $F$. Proposition~\ref{gamma invariant} shows that $F$ is locally convex if and only if $gF$ is convex, where $g \in \Gamma$. Hence, it makes sense to call a face class $\Gamma \cdot F$ locally convex. 

For each face class $\Gamma \cdot F$, we check the neighbouring vertices for any admissible edge flips. If there is an admissible edge flip, then it is performed (replacing $\Gamma \cdot F$ and another face class with two different face classes). This gives a different $\Gamma$--invariant polyhedron with vertices on the light-cone $\mathcal L$, and the entire algorithm starts again.

If there are no admissible edge flips, another face class is checked. Although there are infinitely many faces in the polyhedron, there are only finitely many face classes. The algorithm terminates when there are no more admissible edge flips. The algorithm terminates in finitely many steps (Theorem \ref{algorithm finite}). Moreover, the $\Gamma$--invariant polyhedron in the final iteration is convex (Proposition \ref{algorithm globally convex}).

There is one final step in the algorithm to make the  polyhedron equal to $C$. Even though our polyhedron is convex, since each face class is a triangle, the polyhedron is actually a \emph{triangulation} of $\partial C$. We call an edge of a polyhedron \emph{redundant} if it lies in the interior of a face of $\partial C$. Note that an edge is redundant if and only if the two faces that meet at $e$ are coplanar.

After all admissible edge flips are performed, a cleanup step is applied. If two adjacent faces are coplanar, then their common edge is removed. This is repeated until there are no redundant edges, upon which the polyhedron is equal to the convex hull (Theorem \ref{algorithm correctness}).

Pseudocode for this algorithm is provided below (Algorithm~1).

\begin{algorithm}
\caption{Edge Flipping Algorithm for the Convex Hull Construction}
\begin{algorithmic}[1]
\STATE $T = $ an arbitrary cell decomposition of $\Omega/\Gamma$ into triangles
\STATE $P = \Gamma$-invariant polyhedron induced by $T$. 
\FORALL{$F \in P/\widetilde \Gamma$}
\FORALL{$v$ a neighbouring vertex of $F$}
\IF{$v$ is below $F$}
\STATE{Perform an edge flip on \{$v$, $F$\}}
\STATE{\textbf{go to} line 3}
\ENDIF
\ENDFOR
\ENDFOR
\FORALL{$e$ an edge of $P/\widetilde \Gamma$}
\IF{$F_1$, $F_2$ sharing $e$ are coplanar}
\STATE{Remove $e$ and merge $F_1, F_2$}
\ENDIF
\ENDFOR
\RETURN{$P$}
\end{algorithmic}
\end{algorithm}

%%%%%%%%%%%%%%%%%%%%%

\subsection{Proof of correctness}
\label{sec:correctness}

Given $p, q \in \partial \Omega$ we need to be able to choose light cone representatives $\tilde p,$ $\tilde q$ that either both lie on $\mathcal L^+ = \R^+ \cdot \partial \Omega^+$ or both lie on $\mathcal L^- =\R \cdot \partial \Omega^-.$ We call such lifts \emph{compatible}. (The choice of positive or negative light cone does not affect the notion of convexity since this is preserved by the inversion map.) This is achieved as follows. Choose any lift $\tilde p \in \R^3.$ Consider any non-trivial $A \in \Gamma.$ If $Ap = q$ or $Aq = p,$ then $\tilde{q} = A\tilde{p}$ or $\tilde{q} = A^{-1}\tilde{p}$ respectively is a compatible lift. Otherwise the four points $p,$ $q$, $Ap$, $Aq$ span a rectangle with vertices in  $\partial \Omega.$ Denote $d_1$ and $d_2$ the two diagonals of this rectangle.
Choose any lift $\tilde q \in \R^3$ and denote $\tilde{d}_1$ and $\tilde{d}_2$ the line segments between the chosen lifts amongst $\tilde p$, $\tilde q$, $A\tilde p$, $A\tilde q$ of their endpoints.

\begin{Lemma}
The triangles $[\tilde{d}_1, o]$ and $[\tilde{d}_2, o]$ meet only in $o$ if and only if the lifts $\tilde p$ and $\tilde q$ not compatible.
\end{Lemma}

\begin{proof}
Suppose $\tilde p$ and $\tilde q$ are compatible. Then the triangles meet along a non-degenerate segment. If $\tilde p$ and $\tilde q$ are not compatible, then $\tilde p$ and $-\tilde q$ are compatible. Analysing the effect of the inversion map
on $-\tilde q$ and $-A \tilde q$ in each of the three cases of whether the diagonal from $p$ is $[p, q],$ $[p, Ap]$ or $[p, Aq]$
shows that the triangles only meet in $o.$
\end{proof}

Compatibility is a transitive notion. For the construction, choose orbit representatives $p_1, \ldots, p_k\in \partial \Omega$ for the cusps of $S$ and a lift $\tilde{p}_1\in \R^3$ of $p_1.$ This lift determines the choice of lift of $\Omega$ and hence the positive light cone $\mathcal{L}$. Then for each $j = 2, \ldots k$ choose a light cone representative $\tilde{p}_j$ for $p_j$ compatible with $p_1.$ We usually conjugate $\Gamma$ so that $\tilde{p}_1$ has a simple form, such as $(1,0,0)$, and the only choice left is then to specify a length for each $\tilde{p}_j$ for $j \ge 2,$ giving the $k-1$ degrees of freedom. This completely determines the cell decomposition since if $q$ is in the orbit of some $p_i,$ then Cooper and Long's construction forces $\tilde q = B\tilde{p}_i,$ where
$B \in \Gamma$ is such that $Bp_i = q.$

\begin{Proposition}
\label{gamma invariant}
Point $\widetilde p$ lies below face $F$ $\iff$ $g \widetilde p$ lies below $gF$ for all $g \in \SL(\Omega)$.
\end{Proposition}
\begin{proof}
Let the face $F$ have vertices $\widetilde f_1, \widetilde f_2, \widetilde f_3$ on the light-cone $\mathcal L$, and let their projections onto the boundary $\partial \Omega$ be $f_1, f_2, f_3$ respectively. Let the projection of $\widetilde p$ onto $\partial \Omega$ be $p$ and without loss of generality assume $p, f_1, f_2, f_3$ are in clockwise order around $\partial \Omega$.

Let the segments $pf_2$ and $f_1f_3$ intersect at point $x$. Then $x \in \Omega$ since $\Omega$ is strictly convex. Since $p, x, f_2$ are collinear, the rays through $p, x, f_2$ are coplanar, so segment $\widetilde p \widetilde  f_2$ passes through a point $\lambda x,\, \lambda \in \R^+$. Similarly, segment $\widetilde f_1 \widetilde  f_3$ passes through a point $\mu x,\, \mu \in \R^+$, but $\mu > \lambda$ since this is the only case where $\widetilde p$ is below face $F$. In particular this means that the segment $\widetilde p \widetilde f_2$ intersects the interior of triangle $o \widetilde f_1 \widetilde f_3$, where $o$ is the origin (see Figure~\ref{gamma invariance}).

\begin{figure}
	\centering 
	\includegraphics[width=.55\textwidth]{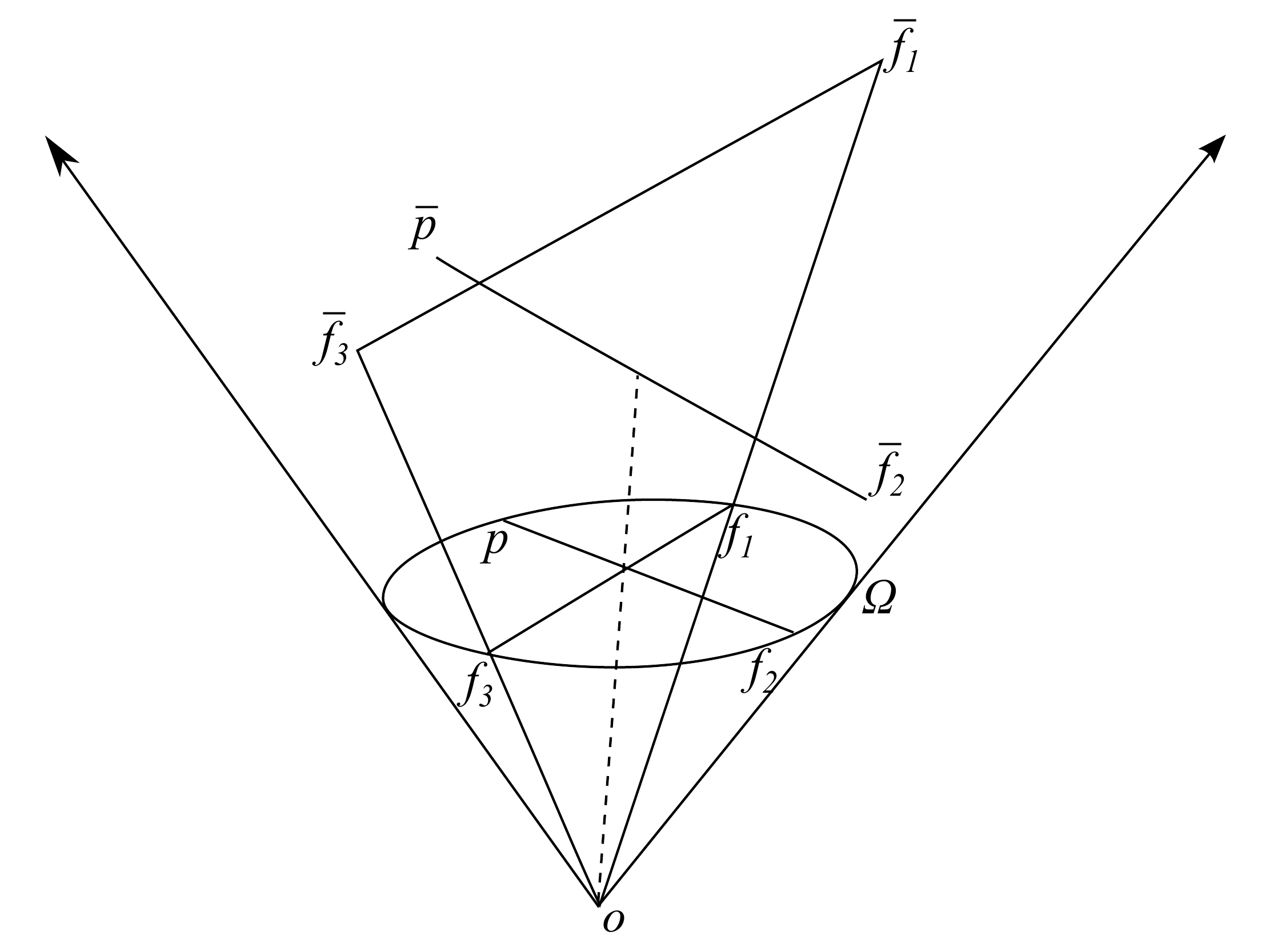} 
	\caption{
	\label{gamma invariance}
	} 
\end{figure}

Hence, $p$ lying below face $F$ is equivalent to the segment $\widetilde p \widetilde f_2$ intersecting the interior of triangle $o \widetilde f_1 \widetilde f_3$ for some ordering $\{\widetilde f_1, \widetilde f_2, \widetilde f_3\}$ of the vertices of $F$. However, the intersection point of the segment and the triangle can be written both as a convex combination of $\widetilde p, \widetilde f_2$ and a convex combination of $o, \widetilde f_1, \widetilde f_3$. This property is linear and preserved by any linear transformation $g \in GL_3(\R)$. 
Hence, if $g \in \SL(\Omega)$, then $g$ preserves the light-cone $\mathcal L$ and $gF$ has its vertices on $\mathcal L$. Moreover, since the convex combination property is linear and preserved by $g$, $\widetilde p$ lies below face $F$ if and only if $g \widetilde p$ lies below $gF$.
\end{proof}

%\begin{Remark}
%Proposition \ref{gamma invariant} gives an alternative way to prove that the convex hull construction in \S\ref{cooper long construction} is $\Gamma$--invariant. If $F$ is a face on the convex hull, then there are no other vertices of the polyhedron below $F$, so the same is true for $gF$ for all $g \in \Gamma$. Hence, $F$ is a face of the convex hull if and only if $gF$ is a face of the convex hull, so the convex hull is $\Gamma$--invariant.
%\end{Remark}

%\begin{Remark}
%An alternative proof is as follows: Extend ray $[o,\widetilde p]$ to intersect $P$ at $x$ and write $x = \lambda \widetilde p$, where $\lambda > 1$. Since $g$ is a linear transformation, $gP$ is a plane and $\lambda g \widetilde p = gx \in gP$. Hence, $g \widetilde p$ and $o$ lie in the same component of $\R^3 \backslash gP$. However, I think I have a counterexample to this proof though.
%\end{Remark}

%\begin{Proposition}
%Point $\widetilde p$ lies below $P$ $\iff$ $g \widetilde p$ lies below $gP$ for all $g \in GL(3,\R)$.
%\end{Proposition}
%\begin{proof}
%Extend ray $[o,\widetilde p]$ to intersect $P$ at $x$ and write $x = \lambda \widetilde p$, where $\lambda > 1$. Since $g$ is a linear transformation, $gP$ is a plane and $\lambda g \widetilde p = gx \in gP$. Hence, $g \widetilde p$ and $o$ lie in the same component of $\R^3 \backslash gP$.
%
%\end{proof}

\begin{Proposition}
\label{algorithm globally convex}
If polyhedron $P$ is locally convex at every face class, then the polyhedron is globally convex. In particular, if for every face $F$, $v$ lies above face $F$ for every neighbouring vertex $v$ of $F$, then for every face $F$, $u$ lies above face $F$ for every vertex $u$ of the entire polyhedron $P$.
\end{Proposition}
\begin{proof}
Suppose that polyhedron $P$ is locally convex at every face, but is not globally convex. There exists a point $\widetilde p$ and a face $\widetilde a \widetilde b \widetilde c$ such that $\widetilde p$ lies below $\widetilde a \widetilde b \widetilde c$. Without loss of generality let the projection of $\widetilde a \widetilde b \widetilde p \widetilde c$ be $a,b,p,c$ in clockwise order around $\partial \Omega$. 

Let the face of $P$ sharing edge $\widetilde b \widetilde c$ with $F$ have vertices $\widetilde b, \widetilde q, \widetilde c$. Since $\widetilde q$ is a neighbouring vertex of $\widetilde a \widetilde b \widetilde c$, $\widetilde q$ lies above the hyperplane passing through $\widetilde a \widetilde b \widetilde c$. In particular, $\widetilde p \widetilde q$ lie on opposite sides of the faces $\widetilde a \widetilde b\widetilde c$ (see Figure~\ref{globally convex}).

\begin{figure}
	\centering 
	\includegraphics[width=.55\textwidth]{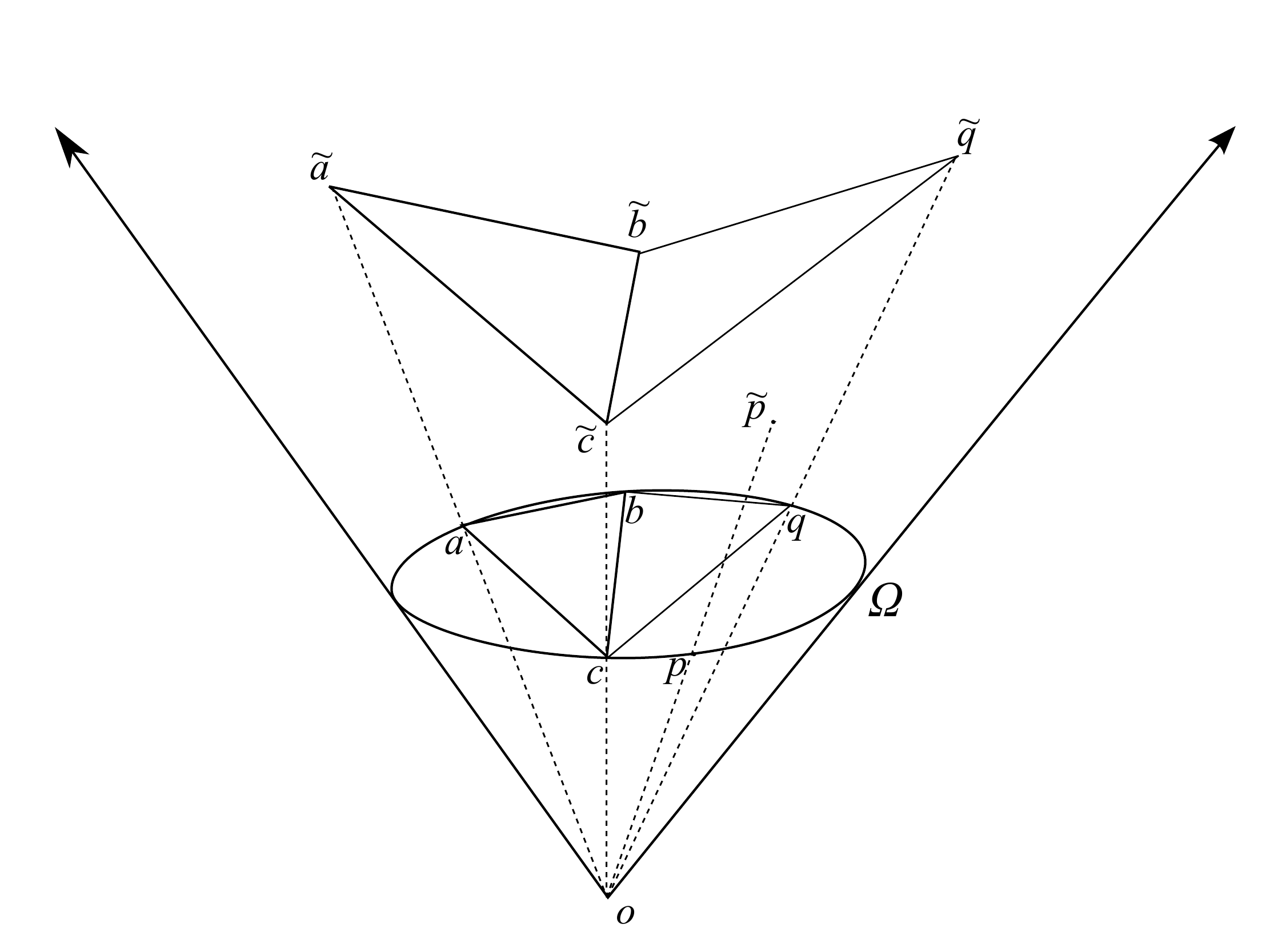} 
	\caption{
	\label{globally convex}
	} 
\end{figure}

Consider hyperplanes through $\widetilde a \widetilde b \widetilde c$ and $\widetilde b \widetilde q \widetilde c$, which  intersect along the line $\widetilde b \widetilde c$. Let the hyperplane through $o \widetilde b \widetilde c$ divide $\R^3$ into two half-spaces, with $T^+$ including $\widetilde p, \widetilde q$ and $T^-$ including $\widetilde a$. Then $\widetilde a \widetilde b \widetilde c$ lies above $\widetilde b \widetilde q \widetilde c$ in the half-space $T^-$, whereas $\widetilde b \widetilde q \widetilde c$ lies above $\widetilde a \widetilde b \widetilde c$ in the half-space $T^+$. Hence, $\widetilde b \widetilde q \widetilde c$ lies above $\widetilde a \widetilde b \widetilde c$ which lies above $p$ in the half-space $T^+$, so $\widetilde p$ lies below $\widetilde b \widetilde q \widetilde c$.

Hence, we have $\widetilde p$ lying below $\widetilde b,\widetilde q,\widetilde c$ instead of $\widetilde a \widetilde b \widetilde c$, where $\widetilde b \widetilde q \widetilde c$ is ``closer'' to $\widetilde p$ than $\widetilde a \widetilde b \widetilde c$. Note that ``closer'' is well defined, as there is a unique path of triangles from $\widetilde a \widetilde b \widetilde c$ to $\widetilde p$ (this is clear if we project onto $\Omega$). Also the path of triangles has finite length. Hence, if we initially assume that $\widetilde p$ is a point which lies below $P$, and $\widetilde a \widetilde b \widetilde c$ is the closest face to $\widetilde p$ such that $\widetilde p$ is below it, then we have another triangle with $\widetilde p$ below it, contradicting the minimality assumption. Hence, if $P$ is locally convex at every face, then $P$ is globally convex.
\end{proof}

\begin{Theorem}
\label{algorithm correctness}
If Algorithm 1 terminates, the output $P$ is the convex hull of the orbit $B$.  
\end{Theorem}
\begin{proof}
For each face class $F \in P/\Gamma$, Algorithm 1 checks that the polyhedron is locally convex at $F$. Hence $P$ is locally convex at every face, and so by Proposition~\ref{algorithm globally convex}, $P$ is convex. Hence the union of the faces of $P$ is equal to $\partial C$. After the cleanup step, there will be no redundant edges and $P = C$.
\end{proof}

\begin{Theorem}
\label{algorithm finite}
Algorithm 1 terminates in finitely many iterations.
\end{Theorem}
\begin{proof}
Let $B$ be the $\Gamma$--invariant vertices of $P$. Define the height of a face $F$ to be the number of points in $B$ below the face. The height of any face is finite since $B$ is discrete, in particular 0 is not an accumulation point. Moreover, the height of a face is $\Gamma$--invariant by Proposition~\ref{gamma invariant}. Hence, we can define the height of the polyhedron $P$ to be the sum of the heights of its face classes $\Gamma \cdot F \in P/\Gamma$.

To prove that Algorithm 1 terminates, it suffices to show that the height of $P$ strictly decreases after every edge flip, since the height of $P$ is always a non-negative integer.

Consider an edge flip $\widetilde a \widetilde c \to \widetilde b \widetilde d$. This occurs only if $\widetilde c$ lies above $\widetilde a \widetilde b \widetilde d$ and $\widetilde a$ lies above $\widetilde b \widetilde c \widetilde d$. Without loss of generality let their projections onto $\partial \Omega$ be $a,b,c,d$ in clockwise order. Let $\widetilde p$ be a point on the light-cone and without loss of generality let the projections $p,a,b,c,d$ be in clockwise order. 

Suppose $\widetilde p$ is below $\widetilde a \widetilde b \widetilde d$. Then consider the triangles $\widetilde a \widetilde b \widetilde d$ and $\widetilde a \widetilde b \widetilde c$. The hyperplane through $o, \widetilde a, \widetilde b$ divides $\R^3$ into two half-spaces, consider only the halfspace which includes $\widetilde c, \widetilde d, \widetilde p$. In this halfspace, we know by local convexity that triangle $\widetilde a \widetilde b \widetilde c$ is above $\widetilde a \widetilde b \widetilde d$, which in turn is above $\widetilde p$. Hence, $\widetilde p$ is below  $\widetilde a \widetilde b \widetilde d$ implies that $\widetilde p$ is also below $\widetilde a \widetilde b \widetilde c$ (see Figure~\ref{finite steps}).

\begin{figure}
	\centering 
	\includegraphics[width=.55\textwidth]{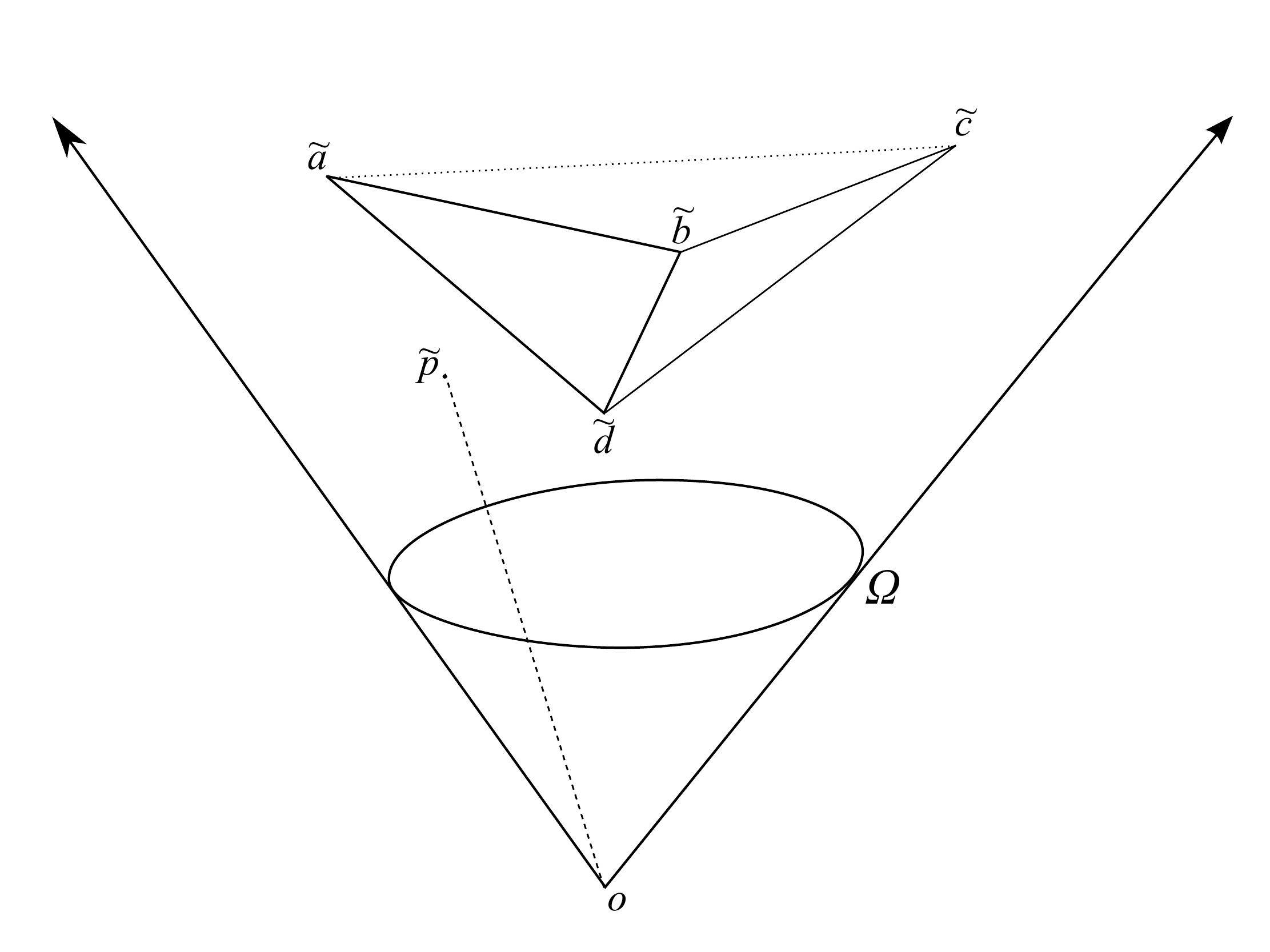} 
	\caption{
	\label{finite steps}
	} 
\end{figure}

Hence, when only considering points with projection between $d,a$ in the strictly convex domain $\Omega$, the points below $\widetilde a \widetilde b \widetilde d$ is a subset of points below $\widetilde a \widetilde b \widetilde c$. By exactly the same argument the points below $\widetilde b \widetilde c \widetilde d$ are a subset of points below $\widetilde a \widetilde c \widetilde d$ when considering this particular part on the boundary $\partial \Omega$. An equivalent result is true the four other arcs $ab, bc$ and $cd$ by rotating the edge labels cyclically. However, the points $\widetilde b, \widetilde d$ contribute to the heights of $\widetilde a \widetilde c \widetilde d, \widetilde a \widetilde b \widetilde c$ before the edge flip, whereas no vertices contribute to the heights of the two triangles $\widetilde a \widetilde b \widetilde d, \widetilde b \widetilde c \widetilde d$. Hence, after every edge flip, the height of the two triangles strictly decreases. The heights of all other triangles not involved in the edge flip stay the same, so the overall result is that the height of $P$ strictly decreases after every edge flip.
\end{proof}

%%%%%%%%%%%%%%%%%%%%%%%%%%
%\newpage
%%%%%%%%%%%%%%%%%%%%%%%%%%

\section{The once-punctured torus}

In this section, the Epstein-Penner decompositions of hyperbolic or strictly convex projective structures on the once-punctured torus are computed using Algorithm 1. We consider two settings that are common in the literature: hyperbolic structures parameterised by representations into $\PSL(2,\R)$; and projective structures parameterised by Goldman coordinates. A third setting, the coordinates of Fock and Goncharov~\cite{FG}, can be found in \cite{HT2015}.

%%%%%%%%%%%%%%%%%%%%%%%%%%

\subsection{Hyperbolic structures}

%%%%%%%%%%%%%%%%%%%%%%%%%%

\textbf{Example 1.}
Take the hyperbolic once-punctured torus given by $U / \Gamma$ where $\Gamma = \langle A, B \rangle$,
$$A = \begin{pmatrix} 2 & 1 \\ 1 & 1 \end{pmatrix}, \; B = \begin{pmatrix} 2 & -1 \\ -1 & 1 \end{pmatrix},$$ with  commutator
$$
[A,B] = ABA^{-1}B^{-1} = \begin{pmatrix} -1 & -6 \\ 0 & -1 \end{pmatrix}.
$$
So $[A,B]$ is parabolic with parabolic fixed point $\infty.$ Following Penner~\cite{penner1987decorated}, we identify $\PSL(2,\R)$ with $\SO^+(1,2)$ using its natural action on symmetric bilinear forms, giving
% the subgroup $\widetilde \Gamma < \SO^+(1,2)$ corresponding to $\Gamma \subset \PSL(2,\R)$ generated by 
$$
\widetilde A = 
\begin{pmatrix}
\frac{7}{2} & 3 & \frac{3}{2} \\
3 & 3 & 1 \\
\frac{3}{2} & 1 & \frac{3}{2}
\end{pmatrix},\;
\widetilde B = 
\begin{pmatrix}
\frac{7}{2} & -3 & \frac{3}{2} \\
-3 & 3 & -1 \\
\frac{3}{2} & -1 & \frac{3}{2}
\end{pmatrix},
$$
and the commutator $[\widetilde{A},\widetilde{B}]$ fixes $v_p = (1,0,-1).$ Our cusped manifold has a fundamental domain with ideal vertices at $p,$ $ Ap,$ $ Bp,$ $ A Bp$ in the Klein model. An initial cell decomposition of the fundamental domain consists of two triangles $\{p, Ap, Bp\}$ and $\{Ap, Bp, ABp\}$, and is the input to the edge flipping algorithm. This corresponds to a $\Gamma$--invariant polyhedron with two face classes $\{p, \widetilde Ap, \widetilde Bp\}$ and $\{\widetilde Ap, \widetilde Bp, \widetilde A \widetilde Bp\}$. Call this polyhedron $P_1$.

\begin{figure}[h]
	\centering
	\subfigure[$P_1$]
	{\includegraphics[width=3.4cm]{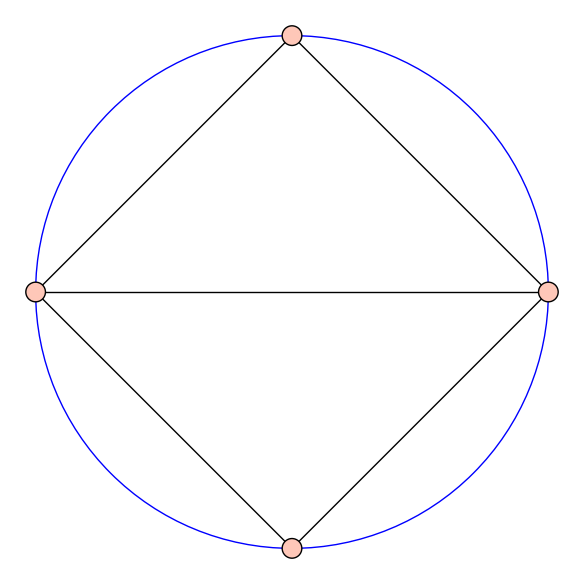}\label{Example 1 P_1 on Gamma}}
	\subfigure[$P_2$]
	{\includegraphics[width=3.4cm]{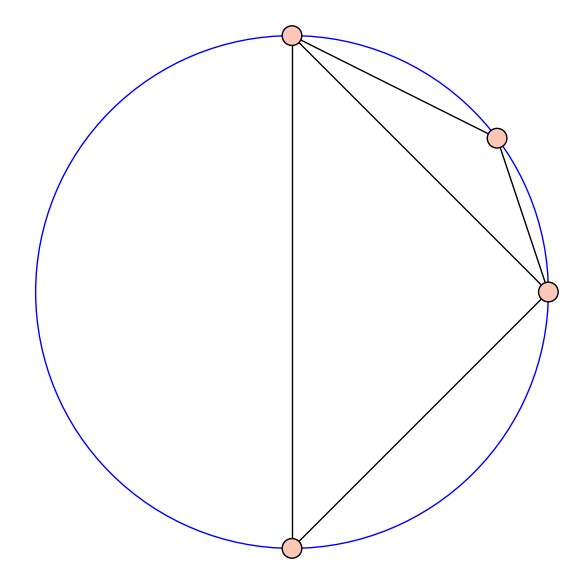}\label{Example 1 P_2 on Gamma}}
	\subfigure[$P_1$]
	{\includegraphics[width=3.4cm]{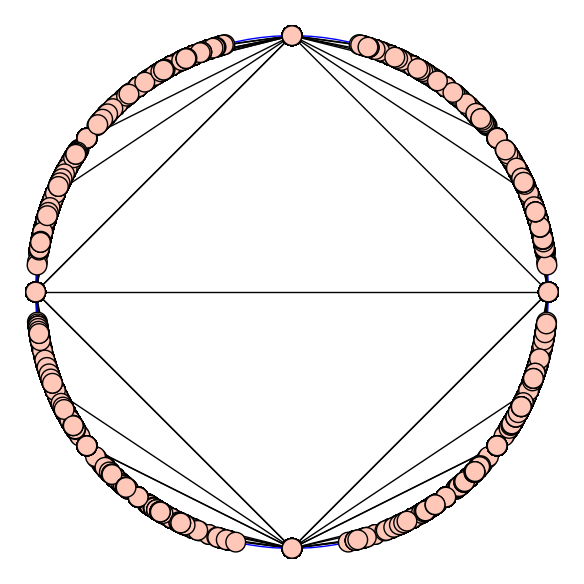}\label{Example 1 P_1}}
	\subfigure[$P_2$]
	{\includegraphics[width=3.4cm]{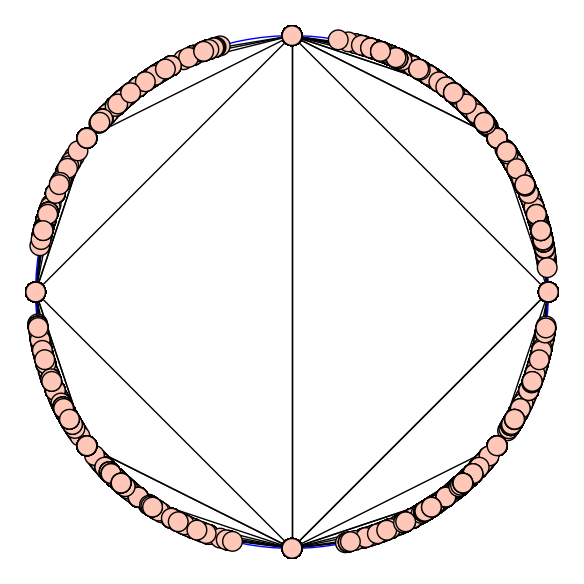}\label{Example 1 P_2}}
\caption{(Example 1) $P_2$ is obtained from $P_1$ by flipping one edge class} 
\end{figure}

The algorithm detects that $\widetilde A\widetilde B v_p$ is below $\{v_p, \widetilde Av_p, \widetilde Bv_p\}$. An edge flip on the triangle $\{\widetilde p,\widetilde  Ap, \widetilde Bp\}$ and its neighbouring vertex $\widetilde A\widetilde Bp$ is performed, giving a new triangulation. The new $\Gamma$--invariant polyhedron is called $P_2$.
The algorithm terminates after one edge flip as all remaining face classes are locally convex. So$P_2$ is the convex hull and induces the canonical cell decomposition of $D$. Fundamental domains for the projections of $P_1$ and $P_2$ onto the Klein model $D$ are shown in Figures \ref{Example 1 P_1 on Gamma} and \ref{Example 1 P_2 on Gamma}, and developed into (part of) cell decompositions of $D$, shown in Figures \ref{Example 1 P_1} and \ref{Example 1 P_2}.

%%%%%%%%%%%%%%%%%%%%%%%%%%

\begin{figure}
	\centering
	 	\subfigure[Projection of fundamental domains of the polyhedra $P_i$ onto the Klein model.]
 		{\includegraphics[width=3.4cm]{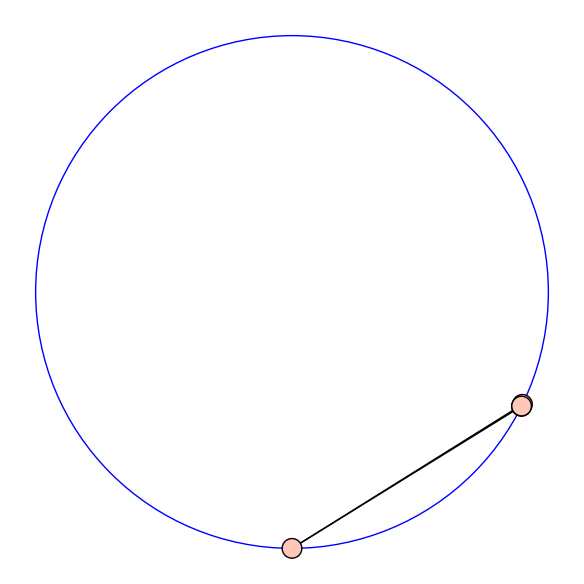} 
		\includegraphics[width=3.4cm]{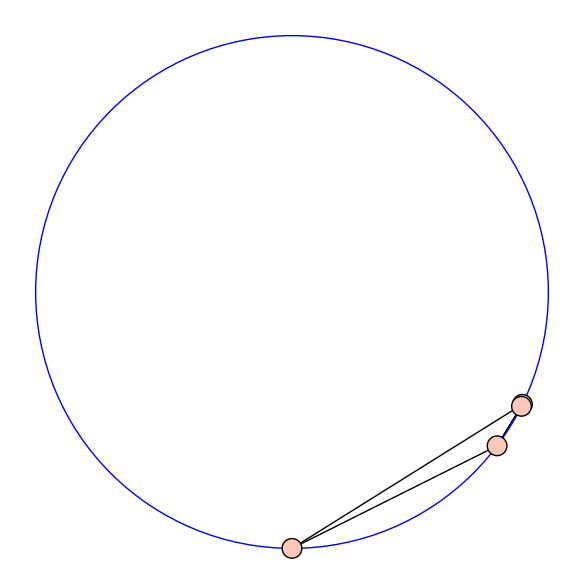} 
		\includegraphics[width=3.4cm]{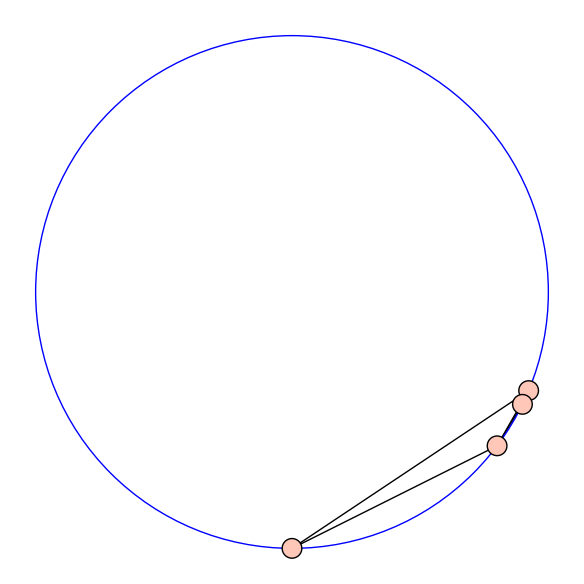} 
		\includegraphics[width=3.4cm]{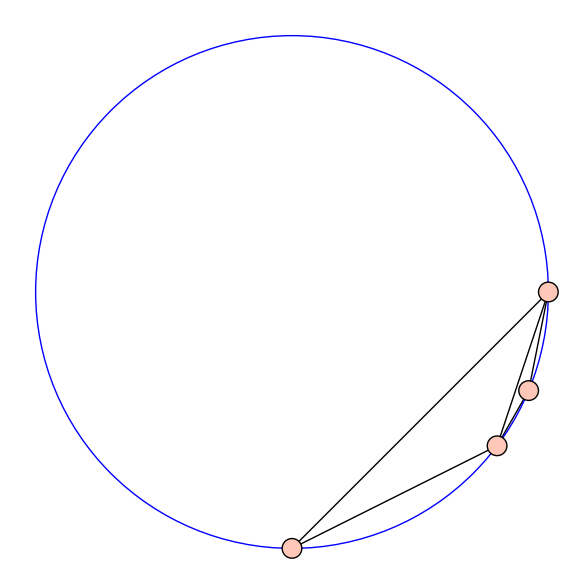}
		\label{hypsecondfun}
		}
		\subfigure[Projection of polyhedra $P_i$ onto the Klein model.]
 		{\includegraphics[width=3.4cm]{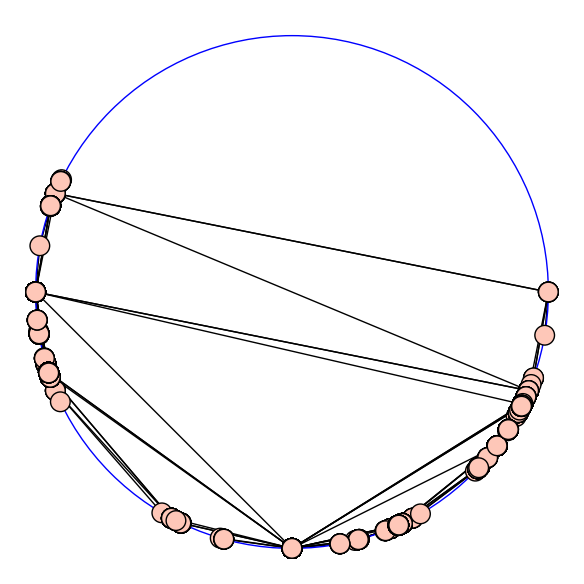} 
		\includegraphics[width=3.4cm]{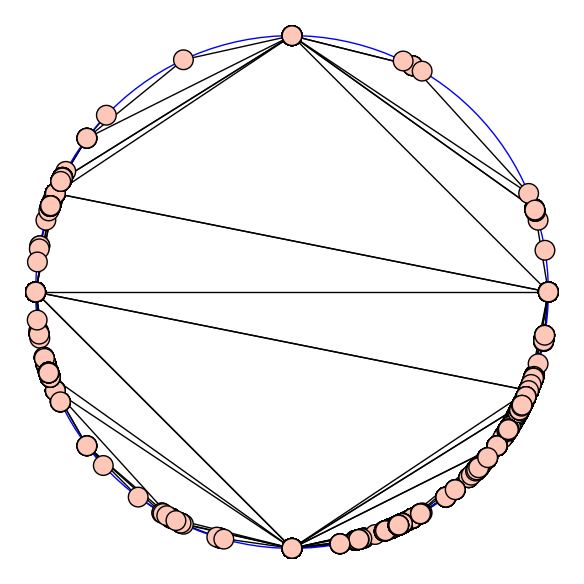} 
		\includegraphics[width=3.4cm]{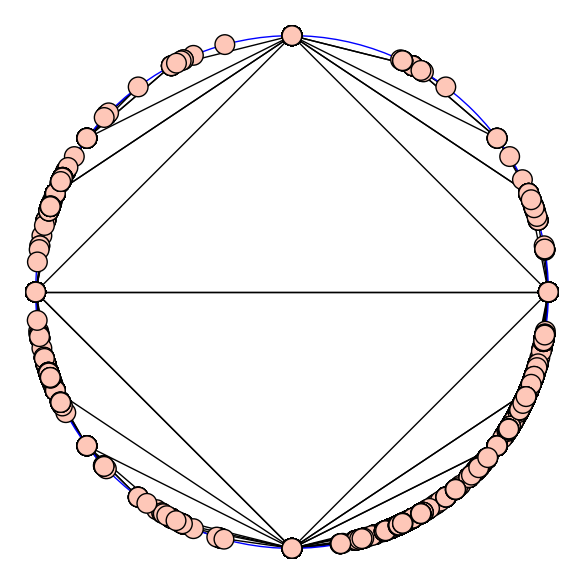} 
		\includegraphics[width=3.4cm]{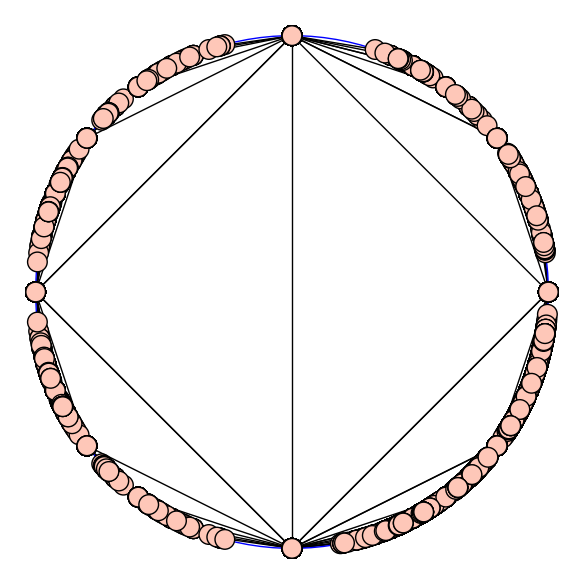}
		\label{hypseconddev}		
		}  
	\caption{(Example 2)  The order in which the polyhedra were visited by the edge flipping algorithm is from left to right.
	\label{hypsecond}
	}
\end{figure}

\textbf{Example 2.}
Take the same once-punctured torus as in Example 1 as a starting point. Let $P_3$ be the polyhedron after applying an edge flip to a \emph{non-admissable} edge. Instead of taking us to the canonical cell decomposition, this takes us further away from it in that there are more edge flips required to get to the convex hull. Continue applying non-admissible edge flips a number of times to give a $\Gamma$--invariant polyhedron $P_n$. Use $P_n$ as the input to Algorithm~1.

Then $P_n$ requires many edge flips for the algorithm to terminate, in fact $P_n$ may need arbitrarily many edge flips. Consider a graph where points are the possible $\Gamma$--invariant polyhedron of the $\Gamma$-orbit of $v_p$, where $v_p = (1,0,-1)$ and $\widetilde \Gamma = \langle \widetilde A, \widetilde B \rangle$. Then there are infinitely many points in this graph, since there are infinitely many choices of a fundamental domain for $D/\Gamma$. Each polyhedron has 3 possible edge flips, so the degree of each vertex in the graph is 3. Hence, there are at most $4^n$ points of distance at most $n$ from the convex hull, so there are infinitely many points with distance $\geq n$ for any $n$. Hence, $P_n$ may need arbitrarily many edge flips to arrive at the convex hull.

An example constructed using this approach is shown in Figure~\ref{hypsecond}, and needs 3 edge flips to reach the convex hull. Since the Epstein-Penner decomposition is canonical, the final cell decomposition is the one computed in Example 1. This is difficult to notice by comparing Figure~\ref{Example 1 P_2 on Gamma} and Figure~\ref{hypsecondfun}, as the fundamental domains of the two figures do not match. However, if we develop the cell decomposition of $D/\Gamma$ to a cell decomposition of $D$, then it is clear that the two cell decompositions do indeed match (see Figure \ref{Example 1 P_2} and Figure~\ref{hypseconddev}).

%%%%%%%%%%%%%%%%%%%%%%%%%%

\begin{figure}[h]
\centering
 	\subfigure[]
 		{\includegraphics[width=3.6cm]{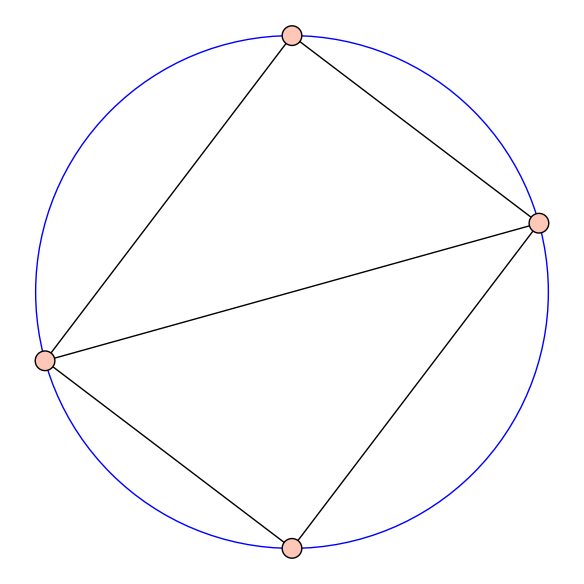}\label{quadminus}}
  	\qquad\qquad
     	\subfigure[]
		{\includegraphics[width=3.6cm]{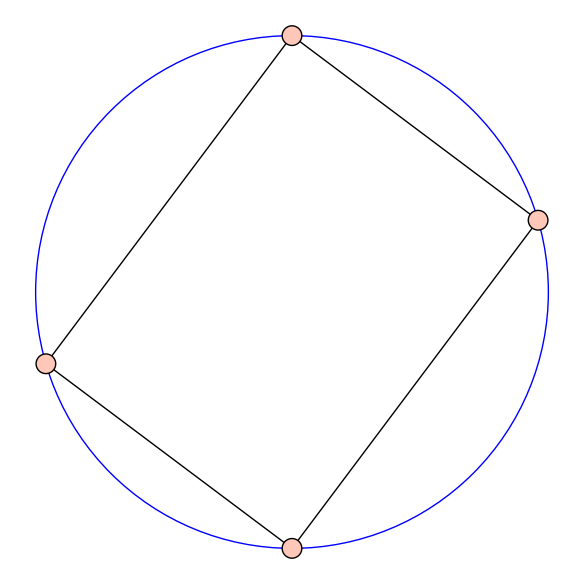}\label{quadflat}}
     	\qquad\qquad
	\subfigure[]
		{\includegraphics[width=3.6cm]{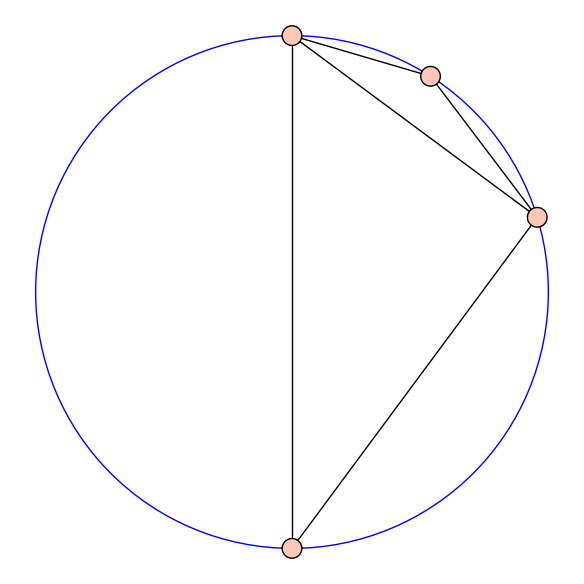}\label{quadplus}}
\caption{(Example 3) A critical point of the parameter space, where the canonical cell decomposition has one cell (shown in the middle) and hyperbolic structures close to this point}
\end{figure}

\textbf{Example 3.}
In this example, the convex hull $C$ has non-triangular faces, and hence the cleanup step is necessary. One possibility where $C$ has a non-triangular face is if  $A, B$ can be found such that $p$, $\widetilde Ap$, $\widetilde Bp$ and $\widetilde A\widetilde Bp$ are coplanar in $\R^3$, this results in a single quadrilateral cell in its canonical cell decomposition. Take the 2-parameter space of complete hyperbolic structures on the once-punctured torus as given in \cite{series1999lectures}:
$$
A = \left(\begin{array}{rr}
\frac{z^{2} + 1}{w} & z \\
z & w
\end{array}\right), \,\,
B = \left(\begin{array}{rr}
\frac{w^{2} + 1}{z} & -w \\
-w & z
\end{array}\right), \,\,\text{where }  z,w \in \R.
$$
Then, the points $p$, $\widetilde Ap$, $\widetilde Bp$ and $\widetilde A\widetilde Bp \in \R^3$ can be calculated in terms of $z$ and $w$. Solving for the case where the four points are coplanar gives the relation 
$$z = \sqrt{1- w^2}.$$
The cell decomposition in Figure~\ref{quadflat} corresponds to the parameters $w = 0.6$, $z = \sqrt{1-w^2} = 0.8$. The resulting cell decomposition is a single ideal quadrilateral, as expected.

Figure~\ref{quadminus} is the cell decomposition corresponding to the parameters $w = 0.6$, $z = 0.799$ whereas Figure~\ref{quadplus} is corresponds to the $w = 0.6$, $z = 0.801$. This small perturbation in opposite direction results in two different cell decompositions. This is expected since the vertices of the quadrilateral $\{\widetilde p, \widetilde Ap, \widetilde Bp, \widetilde A\widetilde Bp\}$ are in a degenerate position, so $w = 0.6,\, z=0.8$ corresponds to a ``critical point'' of the parameter space.

%%%%%%%%%%%%%%%%%%%%%%%%%%

%%%%%%%%%%%%%%%%%%%%%%%%%%

\subsection{Projective structures in Goldman coordinates}
\label{projective once-punctured torus}

\begin{figure}[h]
	\centering 
	\includegraphics[width=9cm]{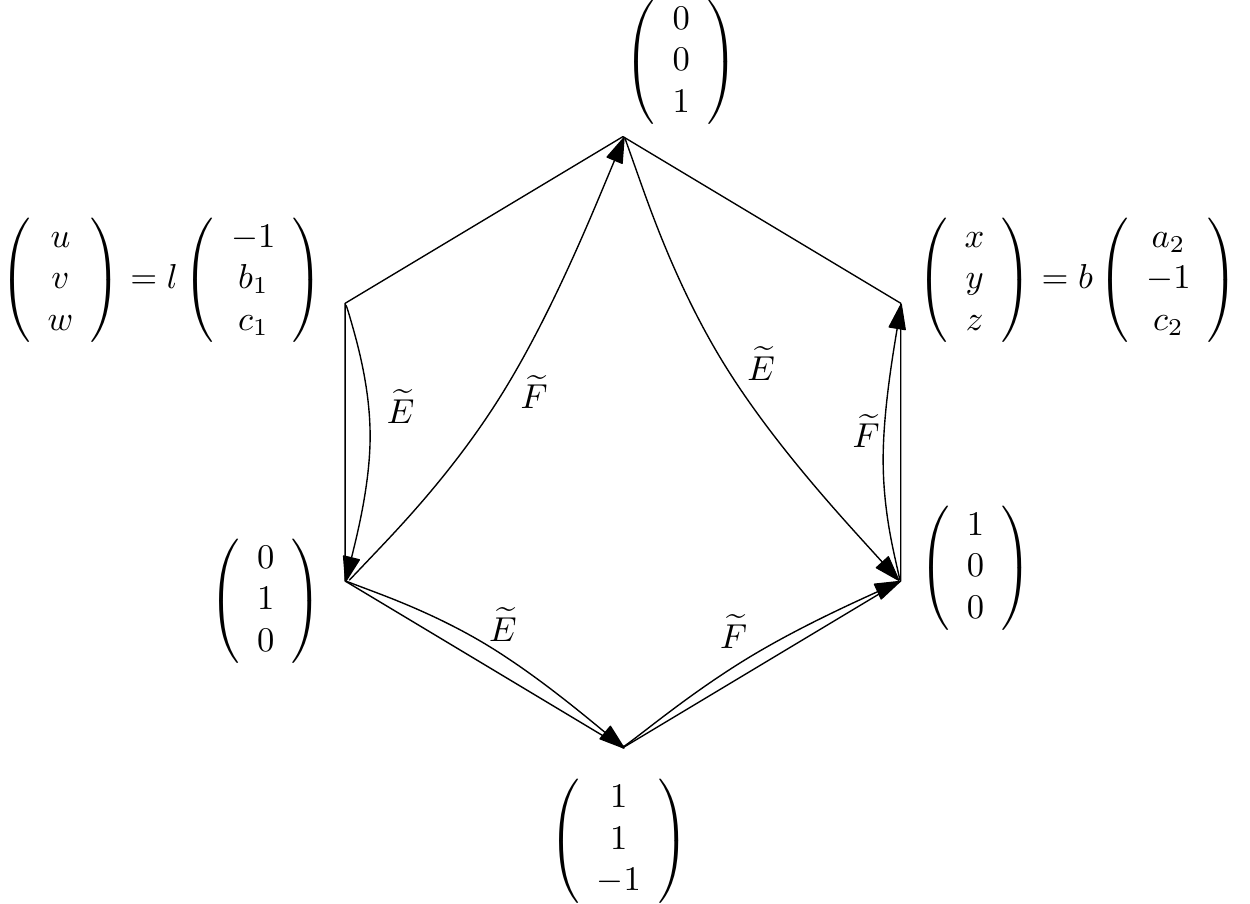} 
	\caption{The action of $\widetilde E$ and $\widetilde F$ on the ideal vertices of $Y$ and two additional points, forming a convex hexagon. 
	\label{marquis hexagon}
	} 
\end{figure}

Following Goldman~\cite{goldman1990convex} and Marquis~\cite{marquis2010espace}, the moduli space of strictly convex projective structures on the once-punctured torus is computed from a fundamental domain an ideal quadrilateral with vertices
\[
\begin{pmatrix} 1 \\ 0 \\ 0 \end{pmatrix},
 \frac 1 {\sqrt 3} \begin{pmatrix} 1 \\ 1 \\ -1 \end{pmatrix},
\begin{pmatrix} 0 \\ 1 \\ 0 \end{pmatrix},
\begin{pmatrix} 0 \\ 0 \\ 1 \end{pmatrix} \in S^2.
\]
The face pairings (cf.\thinspace Figure~\ref{marquis hexagon}) are given by 
\[
\widetilde E 
=
\begin{pmatrix}
ab_1 + \frac {ac_1} {be^2} & a & \frac a {be^2} \\
ab_1 - \frac {be^2} {aa} & a & 0 \\
-ab_1 & -a & 0 \\
\end{pmatrix}
\qquad \text{and} \qquad
\widetilde F 
=
\begin{pmatrix}
a_2b& 0 & a_2b-\frac 1 {be} \\
-b  & 0 & -b \\
c_2b& e & bc_2+e \\
\end{pmatrix}.
\]
Imposing conditions that ensure all structures are convex yields the 6 dimensional semi-algebraic set $\overline{Q}$ consisting of all $(c_1, c_2, b_1, a_2, a, b, e) \in \R^7$ satisfying the inequalities
\[
c_1 > 1,\, c_2 > 1,\, a_2b_1 > 1,\, a, b, e > 0,
\]
and the equation
\[
\begin{aligned}
&a_2 (a^{3} b^{3} b_{1} e^{6} - b^{4} e^{8} + \left(a^{3} b^{3} c_{1} - a^{3} b^{3}\right) c_{2} e^{3} \\
&\hspace{5cm}
+ \left(a^{6} b^{2} b_{1} c_{1} -
a^{6} b^{2} b_{1} - \left(a^{6} b^{2} b_{1} c_{1} - a^{6} b^{2} b_{1}\right) c_{2}\right) e)\\
=\quad& 
a^{3} b^{3} e^{6} + a^{3} b b_{1} e^{5} - a^{6} b_{1} c_{1} c_{2} - 2 \, a^{3} b^{2} e^{4} - b^{2} e^{7} + a^{6} b_{1} c_{1} 
+ a^{3} b c_{1} c_{2} e^{2}\\
&\hspace{5cm}
 + \left(a^{6} b^{2} c_{1} - a^{6} b^{2} - \left(a^{6} b^{2} c_{1} - a^{6} b^{2}\right) c_{2}\right) e
\end{aligned}
\]
The equation can be solved uniquely for $a_2,$ and we will denote the points in the image of the projection of $\overline{Q}$ onto the remaining coordinates by $Q.$ For each $(c_1, c_2, b_1, a, b, e) \in Q$, there is a strictly convex domain $\Omega \subset \RP$ such that $\Omega^+ / \langle \widetilde E, \widetilde F \rangle$ is homeomorphic to the once-punctured torus. Below examples refer to this parameter space.

%%%%%%%%%%%%%%%%%%%%%%%%%%

\begin{figure}[h]
\centering
 	\subfigure[Sequence of edge flips for the projective structure (8,2,3,7,1,2)]
 		{\includegraphics[width=3.6cm]{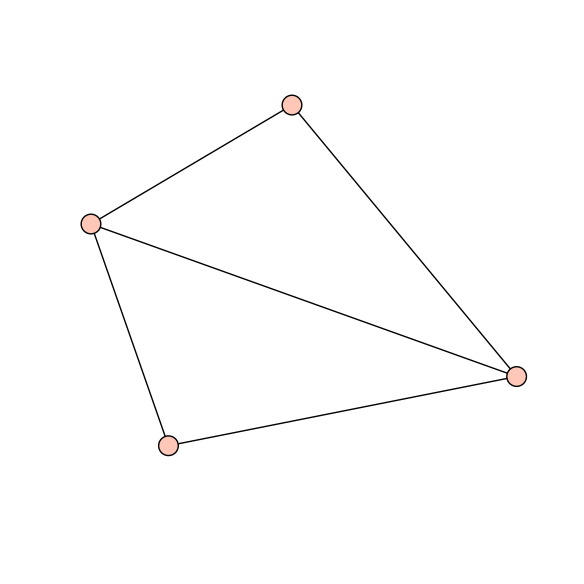} 
	\includegraphics[width=3.6cm]{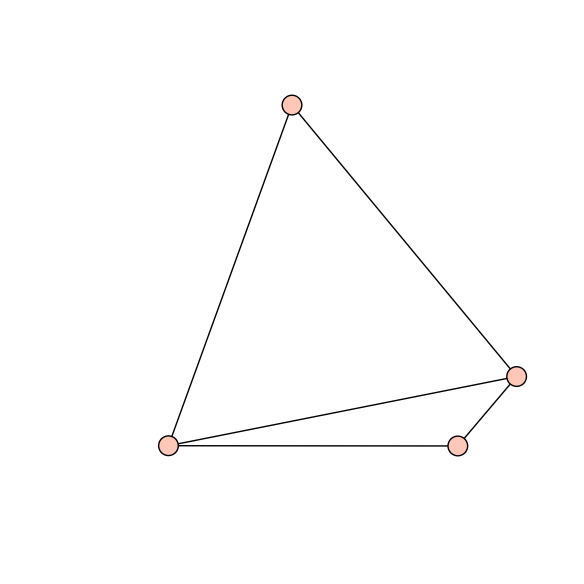} 
	\includegraphics[width=3.6cm]{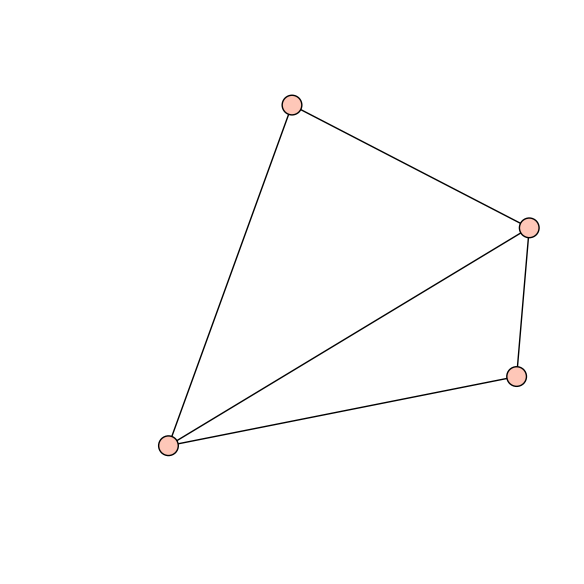} 
	\label{projfirst}
	}
  	\qquad\qquad
     	\subfigure[Developing images of the $\Gamma$--invariant polyhedra]
		{\includegraphics[width=3.6cm]{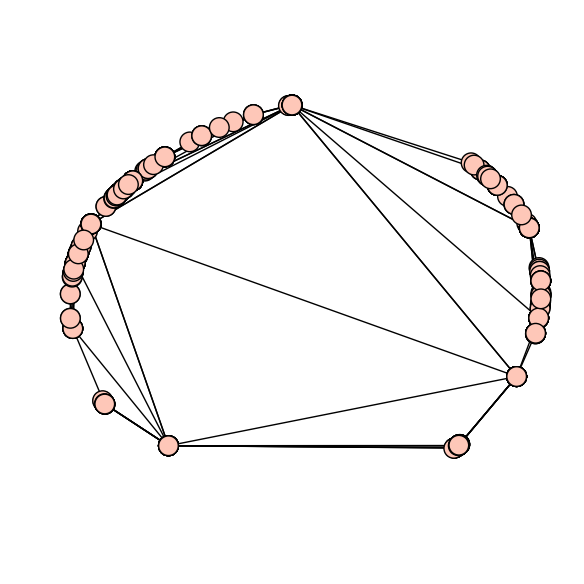} 
	\includegraphics[width=3.6cm]{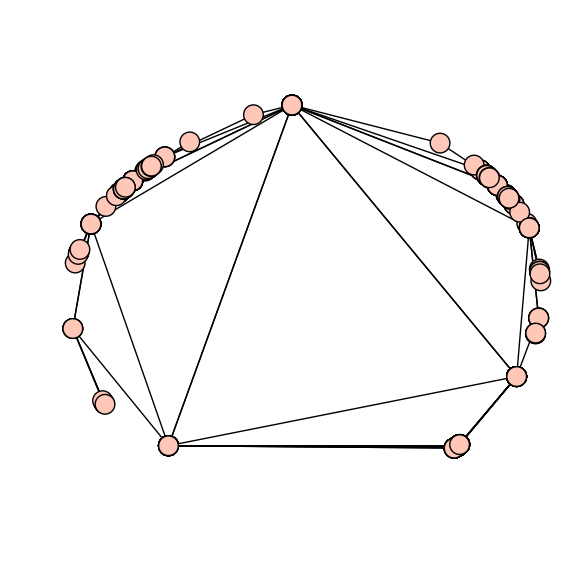} 
	\includegraphics[width=3.6cm]{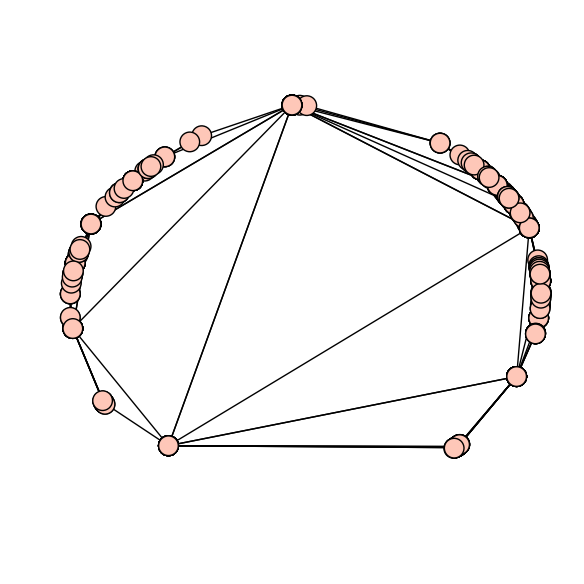}
	\label{projfirstdev}
	}
\caption{(Example 4) $(8,2,3,7,1,2)\in Q$}
\end{figure}

\textbf{Example 4.}
Consider the projective structure corresponding to $(8,2,3,7,1,2) \in Q$. We apply the edge flipping algorithm to find the canonical cell decomposition of $\Omega / \langle \widetilde E, \widetilde F \rangle$. The initial triangulation consists of the two triangles $\{e_1, \widetilde E e_1, \widetilde F e_1\}$ $\{\widetilde E e_1, \widetilde F e_1, \widetilde E \widetilde F e_1\}$. Two edge flips are performed, after which no more neighbouring vertices are below their respective faces, and our algorithm terminates. Figure~\ref{projfirst} show the polyhedra visited by Algorithm 1, in particular, the figure shows its face classes in $S^2$. Figure~\ref{projfirstdev} shows part of the cell decomposition of $\Omega$. The domain $\Omega$ itself is not known. Since the set of parabolic fixed points is dense on $\partial \Omega$, generating many points in the orbit of the cusp gives an approximation to the shape of $\Omega$. In particular, by inspection, the boundary looks strictly convex and has resemblance with an ellipse.

%%%%%%%%%%%%%%%%%%%%%%%%%%

\begin{figure}[h]
\centering
     	\subfigure[Intermediate cell decompositions]
		{\includegraphics[width=4.3cm]{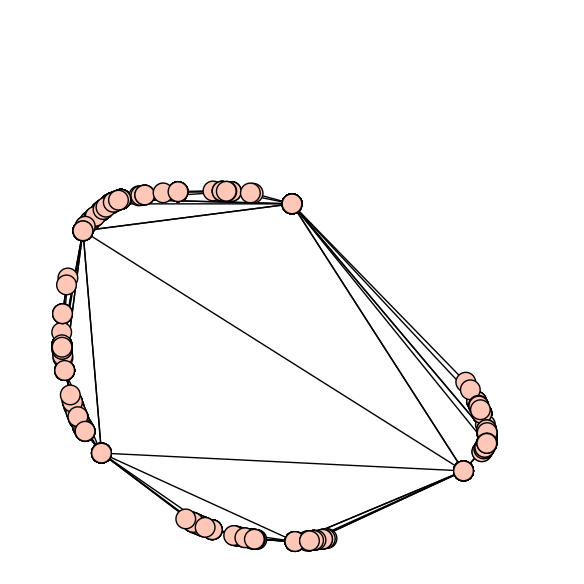} 
	\includegraphics[width=4.3cm]{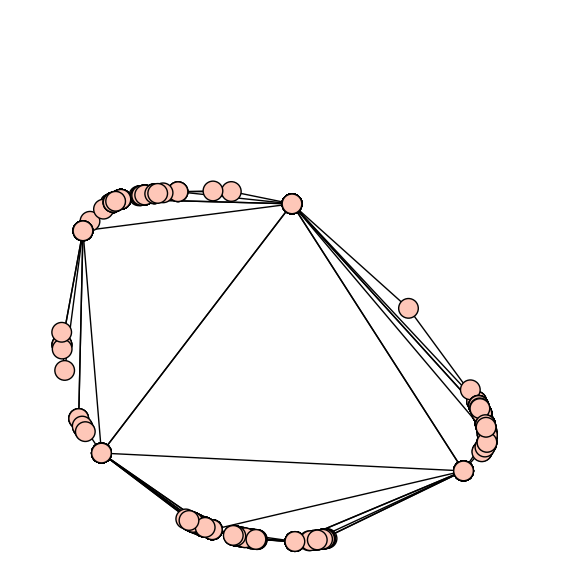} 
	\label{projseconddev}
	}
	\subfigure[An ellipse through 5 points on $\partial \Omega$]
		{\includegraphics[width=5.3cm]{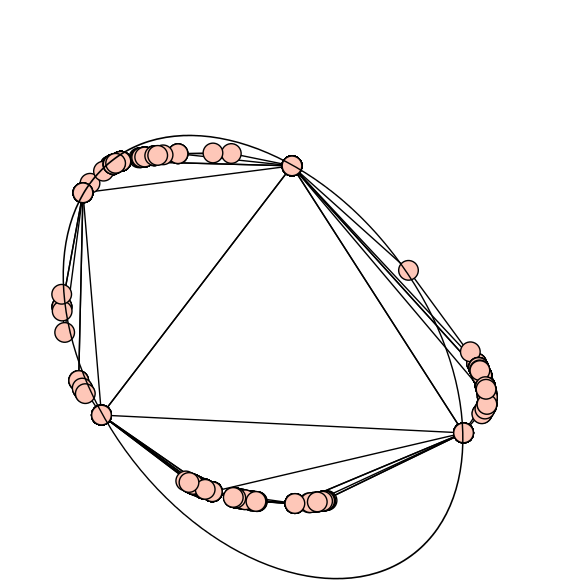}\label{ellipse}}
\caption{(Example 5) $\partial \Omega$ is not an ellipse for (4,7,$\frac 3 2$,1,1,7)}
\end{figure}

\textbf{Example 5.}
The projective once-punctured torus with structure given by $(4,7,\frac 3 2,1,1,7) \in Q$ has face pairings
$$
\widetilde E = 
 \left(\begin{array}{rrr}
\frac{29}{3} & 1 & \frac{2}{3} \\
\frac{11}{2} & 1 & 0 \\
-7 & -1 & 0
\end{array}\right), \qquad
\widetilde F = 
 \left(\begin{array}{rrr}
\frac{85}{69} & 0 & \frac{13}{23} \\
-\frac{3}{2} & 0 & -\frac{3}{2} \\
\frac{21}{2} & 1 & \frac{23}{2}
\end{array}\right) 
$$
with commutator given by
$$
[\widetilde E, \widetilde F]
=
P
 \left(\begin{array}{rrr}
1 & 1 & 0 \\
0 & 1 & 1 \\
0 & 0 & 1
\end{array}\right) 
P^{-1}, \qquad
P = 
 \left(\begin{array}{rrr}
\frac{15654925}{85698} & \frac{182365}{1242} & 0 \\
0 & \frac{580}{207} & 1 \\
0 & -\frac{4205}{414} & 0
\end{array}\right).
$$
We take the light-cone representative $e_1 = (1,0,0)$ of the fixed point of $[\widetilde E, \widetilde F].$ The initial triangulation, and input to the edge flipping algorithm, is the pair of triangles $\{e_1, \widetilde E e_1, \widetilde F e_1\}$ and $\{\widetilde E e_1, \widetilde F e_1, \widetilde E \widetilde F e_1\}$. The algorithm terminates after one edge flip (see Figure~\ref{projseconddev}). The given points in the orbit can be used to certify that $\partial \Omega$ is not an ellipse: We fit an ellipse through 5 points on the boundary, as shown in Figure~\ref{ellipse} (note that it takes 5 points to define an ellipse since two ellipses may intersect at 4 distinct points). The ellipse does not pass through all the points on the boundary. Recall that a projective structure is hyperbolic if and only if the boundary is the central projection of a circle, hence the projective structure defined by $(4,7,\frac 3 2,1,1,7) \in Q$ is not hyperbolic.

%%%%%%%%%%%%%%%%%%%%%%%%%%

\textbf{Example 6.}
If we take two projective structures, for example  $$ (8,2,7,2,4,2),(8,2,3,7,1,2) \in Q,$$ one way to relate the two is to examine the structures in between. In particular, we can look at the cell decompositions of the linear combinations $$(1-\lambda) (8,2,7,2,4,2) + \lambda (8,2,3,7,1,2) \in Q.$$ Figure~\ref{movie} shows the canonincal cell decompositions of 12 evenly spaced sample points along the segment joining $(8,2,7,2,4,2)$ and $(8,2,3,7,1,2)$ in $Q$. Only the canonical cell decompositions are shown. By inspection we can deduce that an edge flip occured along this segment, in particular, between $\lambda = 0.82$ and $\lambda = 0.91$. A binary search with more sample points would improve this estimate further. The parameter $\lambda$ where the edge flip occurs is determined by the zero of polynomial of degree 13 polynomial.

\begin{figure}
	\centering 
	\includegraphics[width=3.6cm]{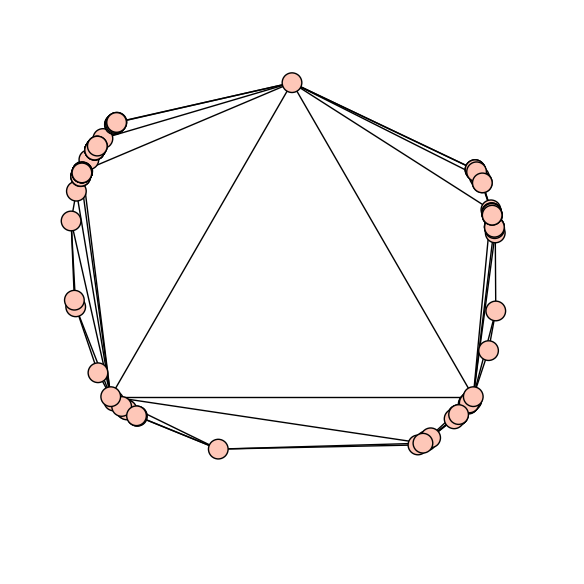} 
	\includegraphics[width=3.6cm]{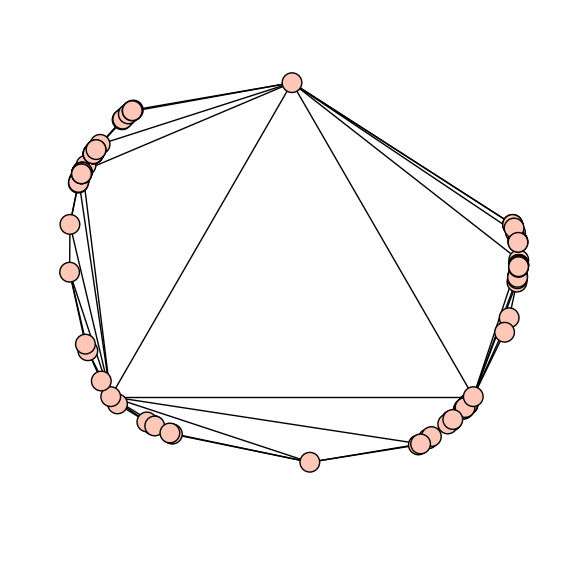} 
	\includegraphics[width=3.6cm]{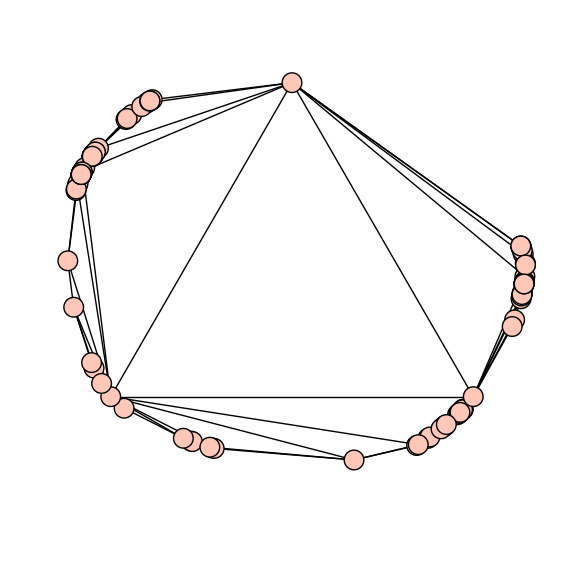} 
	\includegraphics[width=3.6cm]{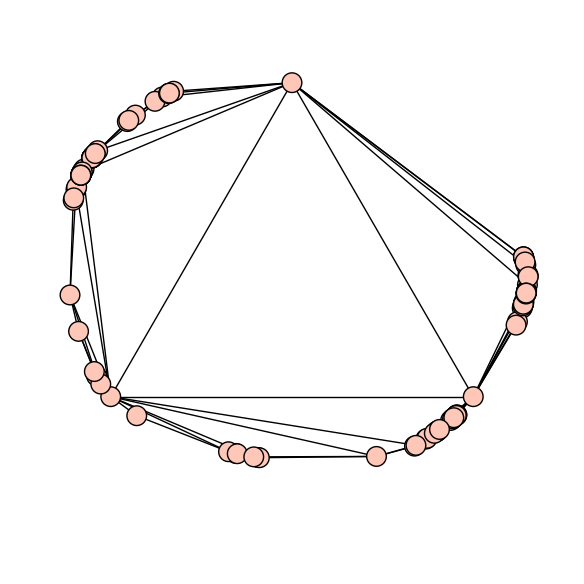} \\
	\includegraphics[width=3.6cm]{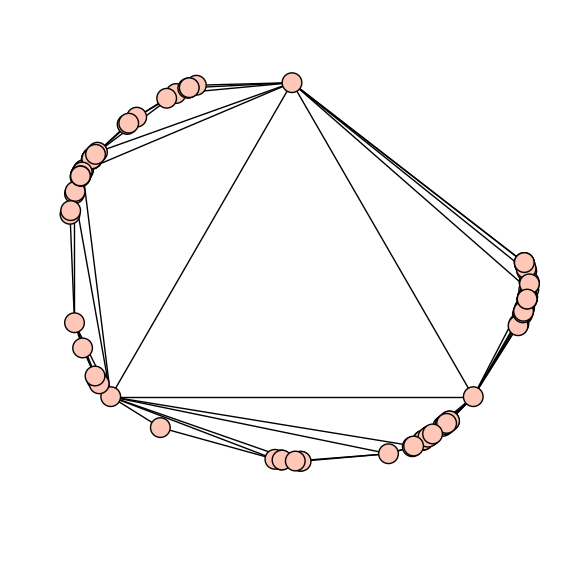}
	\includegraphics[width=3.6cm]{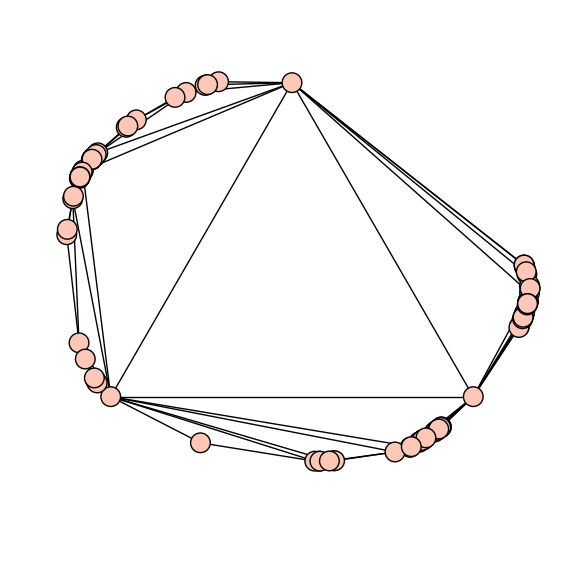} 
	\includegraphics[width=3.6cm]{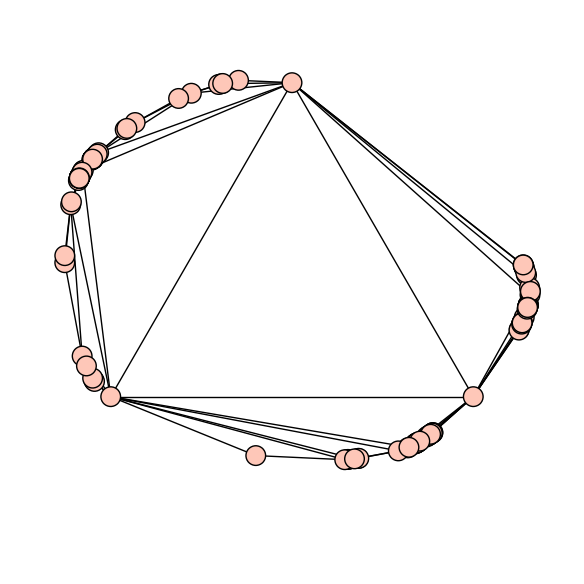} 
	\includegraphics[width=3.6cm]{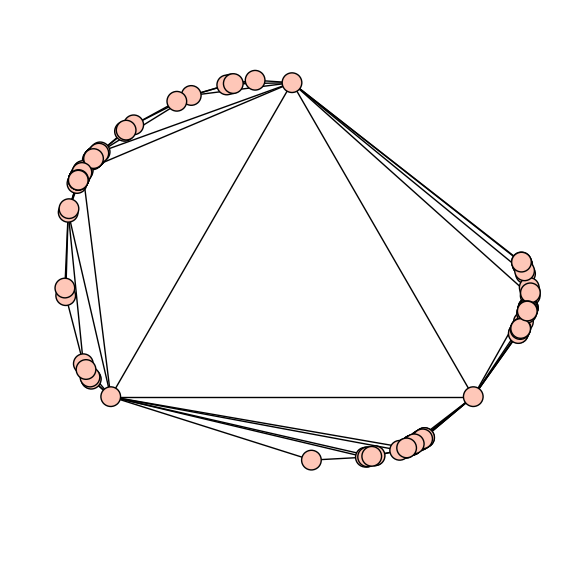} \\
	\includegraphics[width=3.6cm]{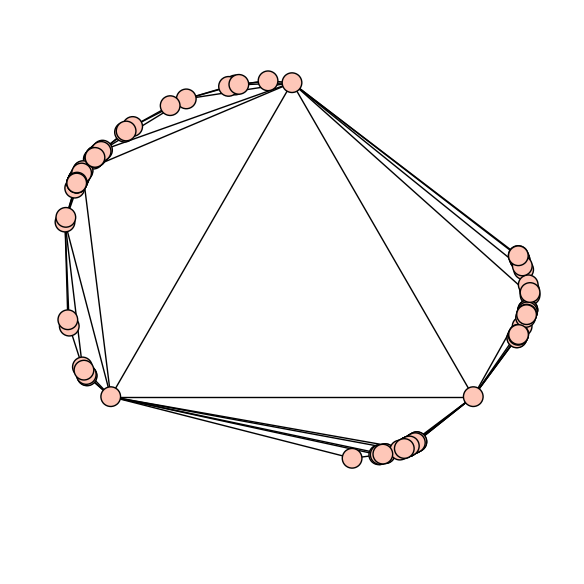} 
	\includegraphics[width=3.6cm]{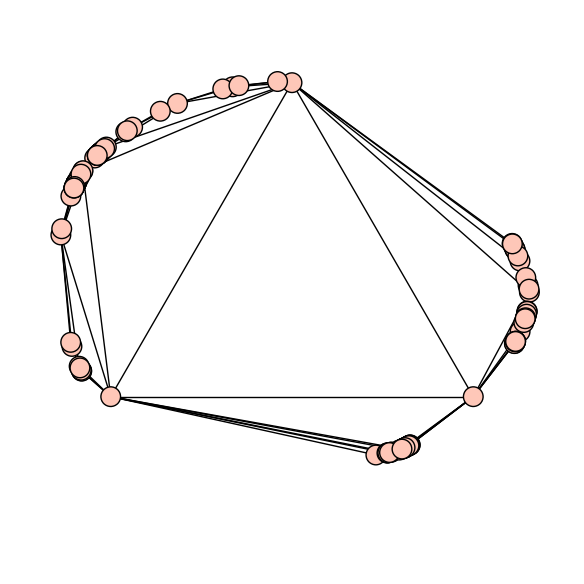} 
	\includegraphics[width=3.6cm]{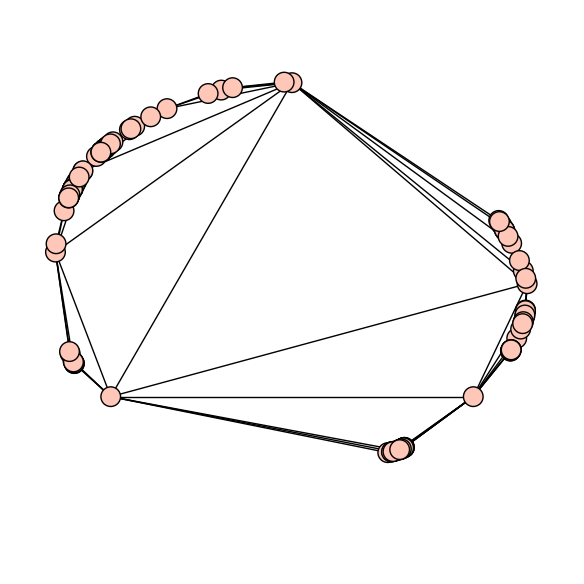} 
	\includegraphics[width=3.6cm]{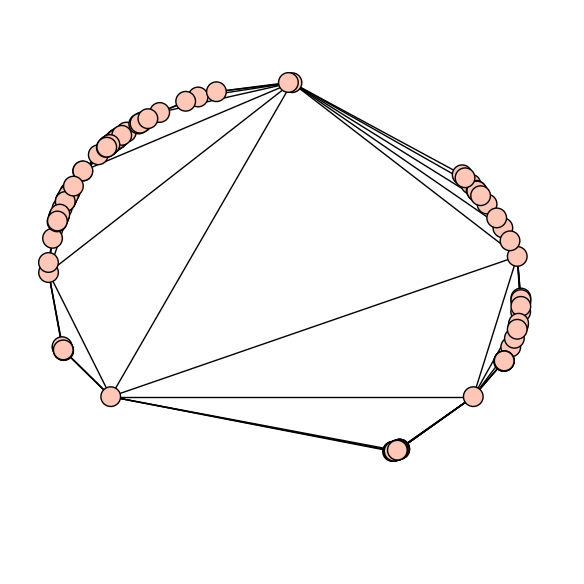} 
	\caption{(Example 6) The cell decompositions corresponding to 12 evenly spaced points along the segment joining $(8,2,7,2,4,2),(8,2,3,7,1,2) \in Q$.
	\label{movie}
	} 
\end{figure}

\subsection*{Acknowledgements}
{This research was partially supported by Australian Research Council grant DP140100158.}

%%%%%%%%%%%%%%%%%%%%%%%%%%
%\newpage
%%%%%%%%%%%%%%%%%%%%%%%%%%

%%%%%%%%%%%%%%%%%%%%%%%%%%%

\address{School of Mathematics and Statistics F07,\\ The University of Sydney,\\ NSW 2006 Australia\\--\\
tillmann@maths.usyd.edu.au\\
s.wong@maths.usyd.edu.au
}

\Addresses

\end{document}